\documentclass[review,12pt]{elsarticle}
\usepackage{amsmath,amssymb,amsthm}
\usepackage{graphicx}
\usepackage{booktabs}
\usepackage[colorlinks=true,linkcolor=black,citecolor=black,urlcolor=blue]{hyperref}

\newtheorem{theorem}{Theorem}
\newtheorem{lemma}{Lemma}
\newtheorem{proposition}{Proposition}
\newtheorem{corollary}{Corollary}
\newtheorem{conjecture}{Conjecture}
\theoremstyle{remark}
\newtheorem{remark}{Remark}

\newcommand{\Z}{\mathbb{Z}}
\newcommand{\F}{\mathbb{F}}
\newcommand{\Cay}{\mathrm{Cay}}
\newcommand{\diam}{\mathrm{diam}}
\newcommand{\kmin}{k_{\min}}

\journal{European Journal of Combinatorics}

\begin{document}
\begin{frontmatter}

\title{Dimension and Order Bounds for Isometric Embeddings of Graphs into
Abelian Cayley Graphs, and the Abelian Dividend}

\author[lesia]{Rigobert Fokam Souop\corref{cor}}
\ead{fokamrigobert@gmail.com}
\author[lesia,saiam]{Laurent Bitjoka}
\cortext[cor]{Corresponding author.}
\address[lesia]{Laboratory of Energy, Signal, Imaging and Automation (LESIA),
University of Ngaound\'er\'e, Ngaound\'er\'e, Cameroon}
\address[saiam]{Laboratory of Scientific Artificial Intelligence and Applied
Mathematics, University of Ngaound\'er\'e, Ngaound\'er\'e, Cameroon}

\begin{abstract}
A finite connected graph $G$ on $n$ vertices embeds isometrically into a Cayley
graph of a finite abelian group; the companion construction \cite{fokam-p1}
realizes such an embedding with binary host dimension at most $n-1$. Here we
quantify how small the host can be. We prove the dimension lower bound
$\kmin(G)\ge\max(\diam(G),\lceil\log_2 n\rceil)$ for binary hosts and the order
lower bound $\nu(G)\ge\max(n,2\,\diam(G))$ for general abelian hosts, and we
show that $\nu(G)=n$ if and only if $G$ is itself an abelian Cayley graph. We
determine the binary dimension exactly for several families: hypercubes,
complete graphs of order $2^t$, and even cycles attain the lower bound; stars
satisfy $\kmin(K_{1,q})=\lceil\log_2 q\rceil+1$ via maximum sum-free sets in
$\Z_2^k$, an exponential gap below the naive and isometric dimensions even on
trees; and odd cycles satisfy $\kmin(C_m)=m-1$ (proved for $m\le 17$ and
reduced in general to a cyclic-interval lemma), so the universal upper bound is
tight. The empirical centerpiece is an exhaustive census of all $995$ connected
graphs on $2\le n\le 7$ vertices, run with a certified search over general
sublattice compactifications. It reveals what we call the \emph{abelian
dividend}: $569$ of the $995$ graphs ($57\%$) admit a certified abelian host
strictly smaller than the best binary host found, and $707$ ($71\%$) admit an
optimal host containing a cyclic factor $\Z_m$ with $m>2$. Compact non-binary
hosts are thus the rule rather than the exception on small graphs, while the
binary host retains its role as the universally guaranteed construction. Only
$17$ of the $995$ graphs attain the order floor $\max(n,2\,\diam)$ exactly;
the floor characterizes highly structured hosts, and the typical dividend is a
modest but strict compression (median $1.6\times$, maximum $9\times$).
\end{abstract}

\end{frontmatter}

\section{Introduction}
\label{sec:intro}

A graph $G$ embeds isometrically into a Cayley graph $\Cay(\Gamma,S)$ of a
finite abelian group $\Gamma$ if there is an injection
$\phi\colon V(G)\to\Gamma$ under which graph distance equals the word metric of
$\Cay(\Gamma,S)$. The study of isometric embeddings of graphs into structured
hosts is classical. Its best-developed chapter concerns hypercubes: Firsov
\cite{firsov1965} first asked which graphs embed isometrically into a Boolean
cube, Graham and Pollak \cite{graham-pollak1971} introduced the addressing
problem for loop switching (resolved by Winkler's proof of the squashed-cube
conjecture \cite{winkler1983}), and Djokovi\'c \cite{djokovic1973} and Winkler
\cite{winkler1984} characterized the isometric subgraphs of hypercubes---the
partial cubes---through the edge relation~$\theta$. Ovchinnikov's survey
\cite{ovchinnikov2008} and the monographs of Imrich, Klav\v{z}ar, and their
coauthors \cite{imrich-klavzar2000,hammack2011} give modern accounts.

Beyond the hypercube, three strands of literature frame our question. First,
isometric embeddings into \emph{Hamming graphs} (products of complete graphs):
Winkler \cite{winkler1984} and Wilkeit \cite{wilkeit1990} characterized the
embeddable graphs and gave polynomial-time algorithms, with the recognition
complexity subsequently sharpened \cite{aurenhammer1995,imrich-klavzar1996}.
Second, the \emph{canonical isometric embedding} of Graham and Winkler
\cite{graham-winkler1985} represents any graph in a product of quotient
graphs; Eppstein's lattice dimension \cite{eppstein2005} measures the minimal
integer-lattice representation. Third, the \emph{scale} and $\ell_1$ theory:
Shpectorov \cite{shpectorov1993} characterized $\ell_1$-graphs via scale
embeddings into hypercubes, with efficient recognition by Deza and Shpectorov
\cite{deza-shpectorov1996}; the monograph of Deza and Laurent
\cite{deza-laurent1997} is the standard reference, and Laurent
\cite{laurent1994} treats hypercube embeddings of distances. The subject
remains active, with recent work on Hamming embeddings of weighted graphs
\cite{berleant2023,sheridan2021} and binary stretch embeddings
\cite{ebrahimi2025}. For the broader metric-graph-theory context---including the structure theory
of graph classes defined by metric conditions, such as retracts of hypercubes
and distance-hereditary graphs \cite{bandelt1984,bandelt-mulder1986}---we
refer to the survey of Bandelt and Chepoi \cite{bandelt-chepoi2008}.

All these hosts are special abelian Cayley graphs: the hypercube is
$\Cay(\Z_2^p,\{e_1,\dots,e_p\})$ and the Hamming graphs are products of
complete circulants. The companion paper \cite{fokam-p1} develops the general
question---isometric embedding into an \emph{arbitrary} abelian Cayley
graph---by a quotient labeling theorem computed with the Smith normal form
(for the role of the Smith normal form in combinatorics see Stanley's survey
\cite{stanley2016}), and proves the universal upper bound: every connected
graph on $n$ vertices embeds isometrically into $\Z_2^{n-1}$. This paper
studies the two natural extremal quantities. For binary hosts, let
\[
\kmin(G)=\min\{k:\ G\hookrightarrow\Cay(\Z_2^k,S)\ \text{isometrically for
some } S\};
\]
for general abelian hosts, let $\nu(G)=\min|\Gamma|$ over all isometric
embeddings into abelian Cayley graphs. The construction of \cite{fokam-p1}
gives $\kmin(G)\le n-1$, hence $\nu(G)\le 2^{\,n-1}$; the question is how far
below these bounds a given graph can go, and what the \emph{typical} position
is.

\paragraph{Results}
We prove (Section~\ref{sec:bounds}) the dimension lower bound
$\kmin(G)\ge\max(\diam(G),\lceil\log_2 n\rceil)$ and the order lower bound
$\nu(G)\ge\max(n,2\,\diam(G))$, together with the characterization
$\nu(G)=n\iff G$ is an abelian Cayley graph (Theorem~\ref{thm:equality}). We
then determine $\kmin$ exactly for several families
(Section~\ref{sec:families}): hypercubes, complete graphs $K_{2^t}$, and even
cycles attain the lower bound, while stars realize an exponential gap,
$\kmin(K_{1,q})=\lceil\log_2 q\rceil+1$, through the extremal theory of
sum-free sets in $\Z_2^k$ (Theorem~\ref{thm:stars}); the relevant extremal
theory of sum-free sets in abelian groups goes back to Rhemtulla and Street
\cite{rhemtulla-street1970} and was completed by Green and Ruzsa
\cite{green-ruzsa2005}. Odd cycles go to the other extreme
(Section~\ref{sec:oddcycles}): $\kmin(C_m)=m-1$, proved for all odd $m\le 17$
and reduced in general to a cyclic-interval lemma (Theorem~\ref{thm:odd},
Conjecture~\ref{conj:interval}), so the universal upper bound cannot be
improved. The window $[\max(\diam,\lceil\log_2 n\rceil),\,n-1]$ is therefore
filled at both ends (Figure~\ref{fig:window}).

The empirical centerpiece (Section~\ref{sec:dividend}) is an exhaustive census
of all $995$ connected graphs on $2\le n\le 7$ vertices \cite{read-wilson1998},
carried out with a certified pipeline whose compactification stage searches
general (including non-diagonal) sublattice folds. It reveals the
\emph{abelian dividend}: a majority of small connected graphs---$569$ of
$995$, or $57\%$---admit a certified abelian host strictly smaller than the
best binary host found, and $71\%$ admit an optimal host with a cyclic factor
$\Z_m$, $m>2$. The binary host retains its structural role as the universally
guaranteed construction and remains optimal for a substantial minority, but
compact non-binary hosts are the rule, not the exception, on small graphs. We
report the methodology, its guarantees, and its limitations in full in
Section~\ref{sec:dividend}, including the caveat that both sides of the
comparison are algorithmic upper bounds; every reported abelian win is,
however, a certified isometric embedding, so the direction of the phenomenon
is not in doubt.

Our motivation is partly applied: isometric abelian hosts equip a graph with
the harmonic analysis of a finite abelian group, connecting graph signal
processing \cite{shuman2013,sandryhaila-moura2013,ortega2018} and spectral
methods \cite{chung1997,hammond2011} to a classical Fourier theory; this
direction is developed in two further companion papers
\cite{fokam-p3,fokam-p4} and in the first author's dissertation
\cite{fokam-diss}.

Throughout, $G=(V,E)$ is finite, connected, with $n=|V|$, $m=|E|$, metric $d$,
and diameter $\diam(G)$; $\Cay(\Gamma,S)$ has the word metric $|\cdot|_S$. We
use the quotient framework of \cite{fokam-p1} freely, in particular the rank
$\rho$ of the (signed) cycle--class matrix, which governs the dimension
$k=t-\rho$ of the quotient host on $t$ classes.

\section{Lower bounds}
\label{sec:bounds}

\subsection{Dimension}

\begin{lemma}[Geodesic independence]\label{lem:geodindep}
If $s_1,\dots,s_d\in S$ are the generators of a geodesic word in
$\Cay(\Z_2^k,S)$, i.e.\ $|s_1+\dots+s_d|_S=d$, then $s_1,\dots,s_d$ are
linearly independent over $\F_2$; in particular all $2^d$ subset sums are
distinct.
\end{lemma}

\begin{proof}
We first show that every sub-multiset of a geodesic word is geodesic. Let
$T\subseteq\{1,\dots,d\}$ and suppose $\bigl|\sum_{i\in T}s_i\bigr|_S=\ell<|T|$,
witnessed by a word $w_1,\dots,w_\ell\in S$ with
$\sum_j w_j=\sum_{i\in T}s_i$. Replacing the letters $\{s_i:i\in T\}$ of the
original word by $w_1,\dots,w_\ell$ leaves the total sum unchanged (the group
is abelian, so letters may be reordered freely) and yields a word of length
$d-|T|+\ell<d$ for $s_1+\dots+s_d$, contradicting geodesicity. Hence
$\bigl|\sum_{i\in T}s_i\bigr|_S=|T|$ for every $T$.

Now suppose two distinct subsets $T\ne T'$ have equal sums. Then the symmetric
difference $T\,\triangle\,T'$ is nonempty and, since every element of
$\Z_2^k$ is an involution, $\sum_{i\in T\triangle T'}s_i=0$, so
$\bigl|\sum_{i\in T\triangle T'}s_i\bigr|_S=0<|T\triangle T'|$, contradicting
the previous paragraph. Distinct subset sums over all $2^d$ subsets is
precisely linear independence over $\F_2$.
\end{proof}

\begin{theorem}[Dimension lower bound]\label{thm:dimlb}
For every connected graph $G$ on $n$ vertices,
$\kmin(G)\ge\max(\diam(G),\lceil\log_2 n\rceil)$.
\end{theorem}

\begin{proof}
Isometric maps are injective, so $2^k\ge n$, giving
$k\ge\lceil\log_2 n\rceil$. For the diameter term, choose $u,v$ with
$d(u,v)=\diam(G)=:d$; isometry yields a geodesic word of length $d$ whose
generators are linearly independent by Lemma~\ref{lem:geodindep}, so the
ambient space has dimension at least $d$.
\end{proof}

\begin{proposition}[Tightness of each term]\label{prop:tight}
\begin{itemize}
\item[(i)] $\kmin(Q_t)=t=\diam(Q_t)=\log_2|V(Q_t)|$ (both terms at once);
\item[(ii)] $\kmin(K_{2^t})=t$, the logarithmic term at diameter $1$,
realized by $K_{2^t}=\Cay(\Z_2^t,\Z_2^t\setminus\{0\})$;
\item[(iii)] $\kmin(C_{2d})=d=\diam(C_{2d})$, realized by the $d$
antipodal-pair classes, which are cuts.
\end{itemize}
\end{proposition}

\begin{proof}
(i) $Q_t$ embeds into itself and the bound gives $k\ge\diam=t$. (ii) The
displayed Cayley graph is $K_{2^t}$: every nonzero element is a generator, so
all pairs of distinct vertices are adjacent; the bound gives
$k\ge\lceil\log_2 2^t\rceil=t$. (iii) The $d$ antipodal classes of $C_{2d}$
are edge cuts, so every cycle crosses each an even number of times, the
cycle--class parity matrix vanishes, and the quotient dimension is $d$; this
is the classical isometric embedding $C_{2d}\hookrightarrow Q_d$
\cite{djokovic1973}, and the bound gives $k\ge\diam=d$.
\end{proof}

The bound is not always attained; Sections~\ref{sec:families}
and~\ref{sec:oddcycles} determine the exact value for two families witnessing
the gap in opposite directions.

\subsection{Order}

The natural currency for general abelian hosts is the order $|\Gamma|$.

\begin{lemma}[Diameter of vertex-transitive graphs]\label{lem:vtdiam}
Every connected vertex-transitive graph $H$ on $N\ge 3$ vertices satisfies
$\diam(H)\le\lfloor N/2\rfloor$.
\end{lemma}

\begin{proof}
Connected vertex-transitive graphs on at least three vertices are
$2$-connected (see, e.g., \cite[Ch.~3]{godsil-royle2001}): a connected
vertex-transitive graph is regular of some degree $r\ge 2$, and its
vertex-connectivity is at least $\tfrac{2}{3}(r+1)>1$. By Menger's theorem any
two vertices $x,y$ lie on a common cycle, i.e.\ are joined by two internally
disjoint paths; the lengths of these paths sum to at most $N$, so the shorter
has length at most $\lfloor N/2\rfloor$, whence
$d_H(x,y)\le\lfloor N/2\rfloor$.
\end{proof}

\begin{theorem}[Order lower bound]\label{thm:orderlb}
For every connected graph $G$ on $n\ge 3$ vertices,
$\nu(G)\ge\max(n,2\,\diam(G))$.
\end{theorem}

\begin{proof}
Injectivity gives $|\Gamma|\ge n$. The host is vertex-transitive and, by
isometry, realizes a distance equal to $\diam(G)$; Lemma~\ref{lem:vtdiam}
gives $\diam(G)\le\lfloor|\Gamma|/2\rfloor$, i.e.\
$|\Gamma|\ge 2\,\diam(G)$.
\end{proof}

\begin{theorem}[Equality at $n$]\label{thm:equality}
$\nu(G)=n$ if and only if $G$ is itself a Cayley graph of an abelian group.
\end{theorem}

\begin{proof}
If $G=\Cay(\Gamma,S)$, the identity embedding gives $\nu(G)\le n$, and
Theorem~\ref{thm:orderlb} gives equality. Conversely, if $\phi$ is isometric
with $|\Gamma|=n$, then $\phi$ is a bijection onto $\Gamma$. Every host edge
$\{\phi(u),\phi(u)+s\}$ joins images of vertices at Cayley distance $1$, hence
at $G$-distance $1$ by isometry, so it is the image of a $G$-edge; and every
$G$-edge maps to a host edge, again by isometry. Thus $\phi$ is a graph
isomorphism $G\cong\Cay(\Gamma,S)$.
\end{proof}

\section{Exact dimensions: stars and the sum-free bound}
\label{sec:families}

One might expect that for graphs with no even cycle the naive embedding is
minimal. This is false, and the smallest counterexample is the star
$K_{1,4}$---a tree, hence even-cycle-free and even a partial cube. The exact
statement is governed by the extremal theory of sum-free sets
\cite{rhemtulla-street1970,green-ruzsa2005}.

\begin{theorem}[Exact dimension of stars]\label{thm:stars}
For every $q\ge 2$, $\kmin(K_{1,q})=\lceil\log_2 q\rceil+1$. In particular
$\kmin(K_{1,4})=3<4=n-1=\mathrm{idim}(K_{1,4})$, so stars realize an
exponential gap between $\kmin$ and both the naive dimension and the isometric
hypercube dimension.
\end{theorem}

\begin{proof}
Normalize an isometric embedding so the center maps to $\mathbf{0}$ (Cayley
graphs are vertex-transitive). The $q$ leaves map to distinct labels
$s_1,\dots,s_q$, and since the only edges are center--leaf, the generating set
is exactly $S=\{s_1,\dots,s_q\}$. Leaves are pairwise at distance $2$, so for
$i\ne j$ we need $s_i+s_j\notin S$ (else the Cayley distance would be $1$):
that is, $S$ is a \emph{sum-free} subset of $\Z_2^k$. Conversely sum-freeness
suffices: $s_i+s_j\ne\mathbf{0}$ (distinct labels) and $s_i+s_j\notin S$ force
$d_{\Cay}(s_i,s_j)\ge 2$, while the two-letter word $s_i,s_j$ realizes $2$.

It remains to find the maximum size of a sum-free set in $\Z_2^k$. If $S$ is
sum-free and $s\in S$, then $S$ and $S+s$ are disjoint subsets of equal size,
so $2|S|\le 2^k$, i.e.\ $|S|\le 2^{k-1}$. This is attained by the set of all
odd-weight vectors, which has size $2^{k-1}$ and is sum-free (a sum of two
odd-weight vectors has even weight); this extremal value is classical
\cite{rhemtulla-street1970} and is the elementary-abelian case of the general
determination of maximal sum-free densities by Green and Ruzsa
\cite{green-ruzsa2005}. Hence $K_{1,q}$ embeds in dimension $k$ iff
$q\le 2^{k-1}$, i.e.\ iff $k\ge\lceil\log_2 q\rceil+1$.
\end{proof}

\begin{corollary}[Minimality of the naive method, scoped]\label{cor:naive}
If $G$ has no even cycle, then no two distinct edges are $\varphi$-related, so
the $\varphi$-quotient embedding of \cite{fokam-p1} coincides with the naive
embedding ($k=n-1$). The naive dimension is therefore minimal within the
$\varphi$-quotient family but, by Theorem~\ref{thm:stars}, not in general.
\end{corollary}

\begin{proof}
Suppose $e=\{u,v\}\;\varphi\;f=\{x,y\}$ with $e\ne f$; by the incidence
property of $\varphi$ \cite{fokam-p1} the four endpoints are distinct. A
shortest $u$--$x$ path, the edge $f$, a shortest $y$--$v$ path (of the same
length as the $u$--$x$ path, by the $\varphi$ equalities), and the edge $e$
close up into a walk of even length $2d(u,x)+2$. In a graph all of whose
cycles are odd, every closed walk of even length must traverse each of its
edges an even number of times (its edge multiset supports no cycle), forcing
the walk to backtrack entirely and hence $\{u,v\}=\{x,y\}$. So all
$\varphi$-classes are singletons and the naive embedding results.
\end{proof}

\section{Odd cycles: the naive bound is tight}
\label{sec:oddcycles}

Stars show $\kmin$ can be exponentially below $n-1$; odd cycles show it can
equal $n-1$, so the universal upper bound is best possible.

Fix an odd cycle $C_m$, $m=2d+1$, with edges $e_0,\dots,e_{m-1}$ in cyclic
order and edge generators $s_0,\dots,s_{m-1}$ (the generator of $e_i$ is
$\lambda(v_i)\oplus\lambda(v_{i+1})$); going once around gives
$\sum_i s_i=0$. Encode dependencies by the \emph{dependency code}
\[
D=\Bigl\{x\in\F_2^m:\ \sum_{i:x_i=1}s_i=0\Bigr\},\qquad \mathbf{1}\in D,
\]
a linear code with $\dim D=m-\mathrm{rank}\{s_i\}$. For a cyclic interval
(arc) $W\subseteq\Z_m$ the endpoint labels differ by $\sum_{i\in W}s_i$, and
for any $x\in D$ the subset $W\triangle x$ sums to the same element. Isometry
forces
\begin{equation}\label{eq:arc}
\min_{x\in D}\,|W\triangle x|=\min\bigl(|W|,\,m-|W|\bigr)
\quad\text{for every arc } W. \tag{$\dagger$}
\end{equation}

\begin{lemma}[Cyclic interval lemma]\label{lem:interval}
Let $m$ be odd and $x\subseteq\Z_m$ with $x\notin\{\emptyset,\Z_m\}$. Then
there is an arc $W$ with $|W\triangle x|<\min(|W|,m-|W|)$. This holds in the
covering-arc regime (proved below) and has been verified exhaustively by
computer for all odd $m\le 17$.
\end{lemma}

\begin{proof}[Proof in the covering-arc regime, and verification status]
Arcs are closed under complementation and $|W^c\triangle x^c|=|W\triangle x|$,
so assume $w:=|x|\le d$. If the support of $x$ lies in an arc $W$ of length
$L\le d$, then $|W\triangle x|=L-w<L=\min(L,m-L)$, since $L\le d$ gives
$m-L\ge d+1>L$. The remaining regime---supports of weight at most $d$ so
spread that every covering arc is longer than $d$---was checked exhaustively
for all odd $m\le 17$ ($2^{17}-2$ subsets at $m=17$), with no counterexample.
\end{proof}

\begin{theorem}[Odd cycles require the naive dimension]\label{thm:odd}
For every odd $m\le 17$ (and every odd $m$ for which
Lemma~\ref{lem:interval} holds), $\kmin(C_m)=m-1$.
\end{theorem}

\begin{proof}
The upper bound is the construction of \cite{fokam-p1} (all classes
singletons, one cycle relation, $k=m-1$). For the lower bound, suppose
$\mathrm{rank}\{s_i\}\le m-2$; then $\dim D\ge 2$, so $D$ contains some
$x\notin\{\mathbf{0},\mathbf{1}\}$. By Lemma~\ref{lem:interval} some arc $W$
has $|W\triangle x|<\min(|W|,m-|W|)$, the graph distance of $W$'s endpoints;
the generator subset $W\triangle x$ sums to their label difference, a word
shorter than their distance, violating~\eqref{eq:arc}. Hence
$\mathrm{rank}\{s_i\}=m-1$ and $k\ge m-1$. Independently, exhaustive search
confirms $\kmin(C_5)=4$ and $\kmin(C_7)=6$ directly.
\end{proof}

\begin{conjecture}\label{conj:interval}
Lemma~\ref{lem:interval} holds for every odd $m$; consequently
$\kmin(C_m)=m-1$ for every odd cycle.
\end{conjecture}

\subsection{Synthesis}

Table~\ref{tab:families} and Figure~\ref{fig:window} summarize. The value
$\kmin$ ranges over the whole window
$[\max(\diam,\lceil\log_2 n\rceil),\,n-1]$, both ends are achieved, and the
position within the window is governed by how much linear dependency the cycle
structure permits among edge generators---exactly the rank $\rho$ that the
quotient framework of \cite{fokam-p1} computes.

\begin{table}[t]
\centering
\caption{Exact binary dimension for several families against the lower bound
of Theorem~\ref{thm:dimlb}.}
\label{tab:families}
\begin{tabular}{lccc}
\toprule
family & $\kmin$ & lower bound & status\\
\midrule
hypercube $Q_t$ & $t$ & $t$ & attained\\
complete $K_{2^t}$ & $t$ & $t$ & attained\\
even cycle $C_{2d}$ & $d$ & $d$ & attained\\
Petersen & $4$ & $4$ & attained\\
star $K_{1,q}$ & $\lceil\log_2 q\rceil+1$ & $\lceil\log_2(q{+}1)\rceil$
 & gap $\le 1$\\
odd cycle $C_{2d+1}$ & $2d$ & $\max(d,\lceil\log_2(2d{+}1)\rceil)$
 & gap $\approx d$\\
\bottomrule
\end{tabular}
\end{table}

\begin{figure}[t]
\centering
\includegraphics[width=.95\linewidth]{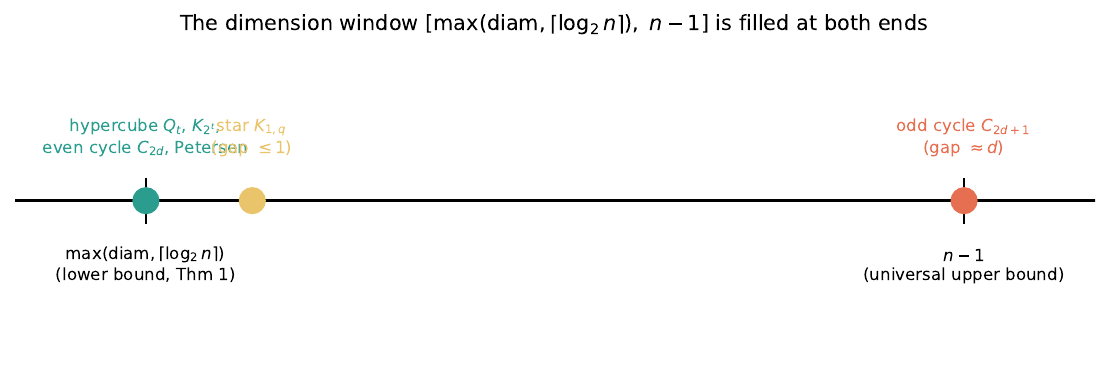}
\caption{The dimension window. $\kmin$ spans
$[\max(\diam,\lceil\log_2 n\rceil),\,n-1]$; hypercubes, complete graphs, even
cycles, and the Petersen graph sit at the lower end, stars one step above it,
and odd cycles at the upper end.}
\label{fig:window}
\end{figure}

\section{The abelian dividend}
\label{sec:dividend}

The most informative experiment is a census. We ran the embedding pipeline of
\cite{fokam-p1}---with the compactification stage searching \emph{general}
finite-index sublattices of the free part of the universal group, enumerated
in Hermite normal form and each candidate certified by an exact breadth-first
distance check---on all $995$ connected graphs on $2\le n\le 7$ vertices
\cite{read-wilson1998}, and compared each best certified abelian host with the
best binary host produced by the binary ($\varphi$-quotient) pipeline.

\subsection{Methodology and guarantees}
\label{subsec:methodology}

Three properties of the experiment should be stated before its numbers.
First, every reported abelian host is \emph{certified}: the pipeline verifies
all $\binom{n}{2}$ distances against a breadth-first search of the actual
finite Cayley graph, so a reported host is a genuine isometric embedding, not
a heuristic estimate. Second, both sides of the comparison are nevertheless
\emph{algorithmic upper bounds}: the ``best binary host'' is the output of the
binary pipeline (an upper bound on $2^{\kmin}$) and the ``best abelian host''
is the smallest certified host found by the portfolio (an upper bound on
$\nu$). ``Strictly smaller'' therefore means that the abelian pipeline
certifiably beat the binary pipeline; since each abelian win is a certified
embedding, $\nu$ really is below the binary pipeline's host for those graphs,
so the \emph{direction} of the phenomenon is robust even though individual
optima are not claimed. Third, the enumeration is conservative: for a small
number of computationally hard instances the sublattice search was restricted
(full Hermite-normal-form enumeration for free rank at most two, diagonal
folds beyond), which can only under-count compact hosts; the dividend reported
below is therefore, if anything, understated. Complete per-graph records and
the runner scripts accompany the paper.

\subsection{Results}

\begin{figure}[t]
\centering
\includegraphics[width=.98\linewidth]{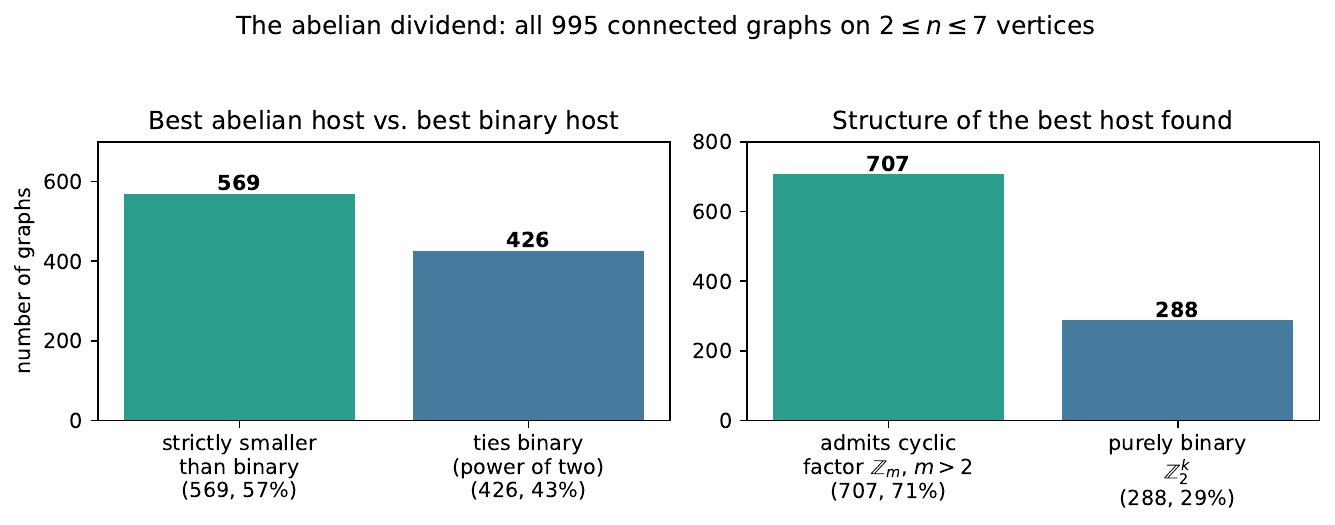}
\caption{The abelian dividend over all $995$ connected graphs on
$2\le n\le 7$ vertices. Left: a majority of graphs gain strictly from
non-binary hosts. Right: $71\%$ admit a non-involutive cyclic factor in the
best host found.}
\label{fig:census}
\end{figure}

The distribution is shown in Figure~\ref{fig:census}:
\begin{itemize}
\item $569$ graphs ($57.2\%$) receive a certified abelian host strictly
smaller than their best binary host;
\item the remaining $426$ ($42.8\%$) tie---their best host found is a power
of two;
\item $707$ graphs ($71.1\%$) admit a best host containing a cyclic factor
$\Z_m$ with $m>2$; only $288$ ($28.9\%$) are purely binary.
\end{itemize}
When the dividend is strict, its typical size is modest: the median
compression over the $569$ winners is $1.6\times$, with a maximum of
$9\times$, and eleven graphs gain a factor of four or more. Only $17$ of the
$995$ graphs attain the order floor $\max(n,2\,\diam)$ of
Theorem~\ref{thm:orderlb} exactly (Figure~\ref{fig:boundsfit}): the floor is
the signature of highly structured hosts (cycles, paths, circulants, and the
graphs of Theorem~\ref{thm:equality}), not of the typical graph.
Figure~\ref{fig:gallery} displays six certified non-binary embeddings with
their explicit group labels; each was re-verified against a full
breadth-first computation of the host's metric, and four of the six attain
the order floor.

\begin{figure}[t]
\centering
\includegraphics[width=.75\linewidth]{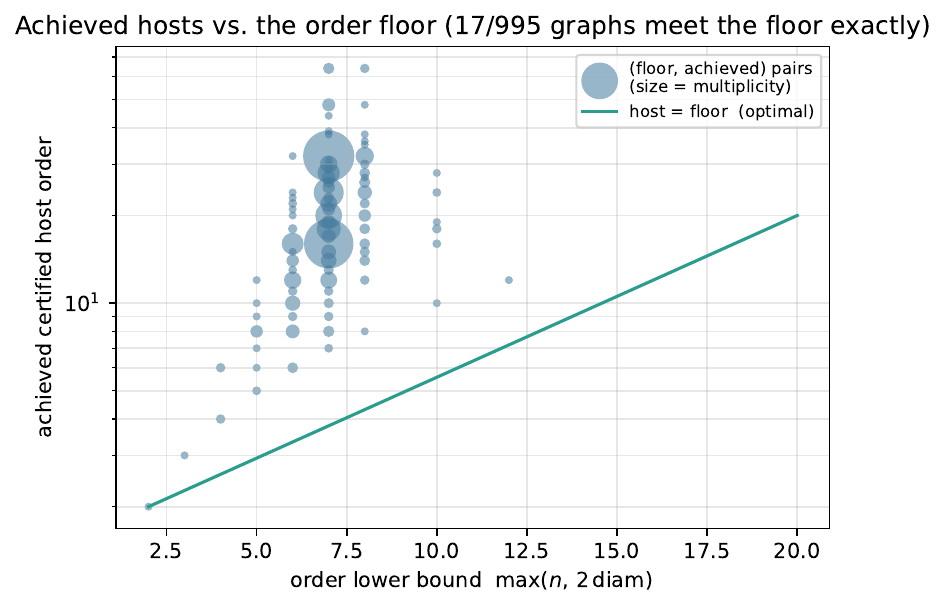}
\caption{Achieved certified host orders against the order floor
$\max(n,2\,\diam)$ across the census. The diagonal is optimality;
$17/995$ graphs attain it.}
\label{fig:boundsfit}
\end{figure}

\begin{figure}[t]
\centering
\includegraphics[width=.7\linewidth]{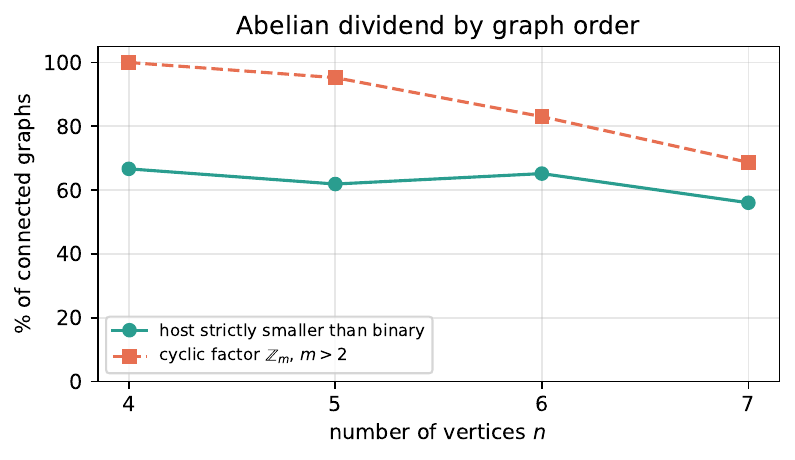}
\caption{The dividend by graph order. The strict-dividend fraction is stable
around $56$--$65\%$ on $4\le n\le 7$; whether it persists, grows, or declines
at larger $n$ is an open empirical question.}
\label{fig:byn}
\end{figure}

\begin{figure}[t]
\centering
\includegraphics[width=.98\linewidth]{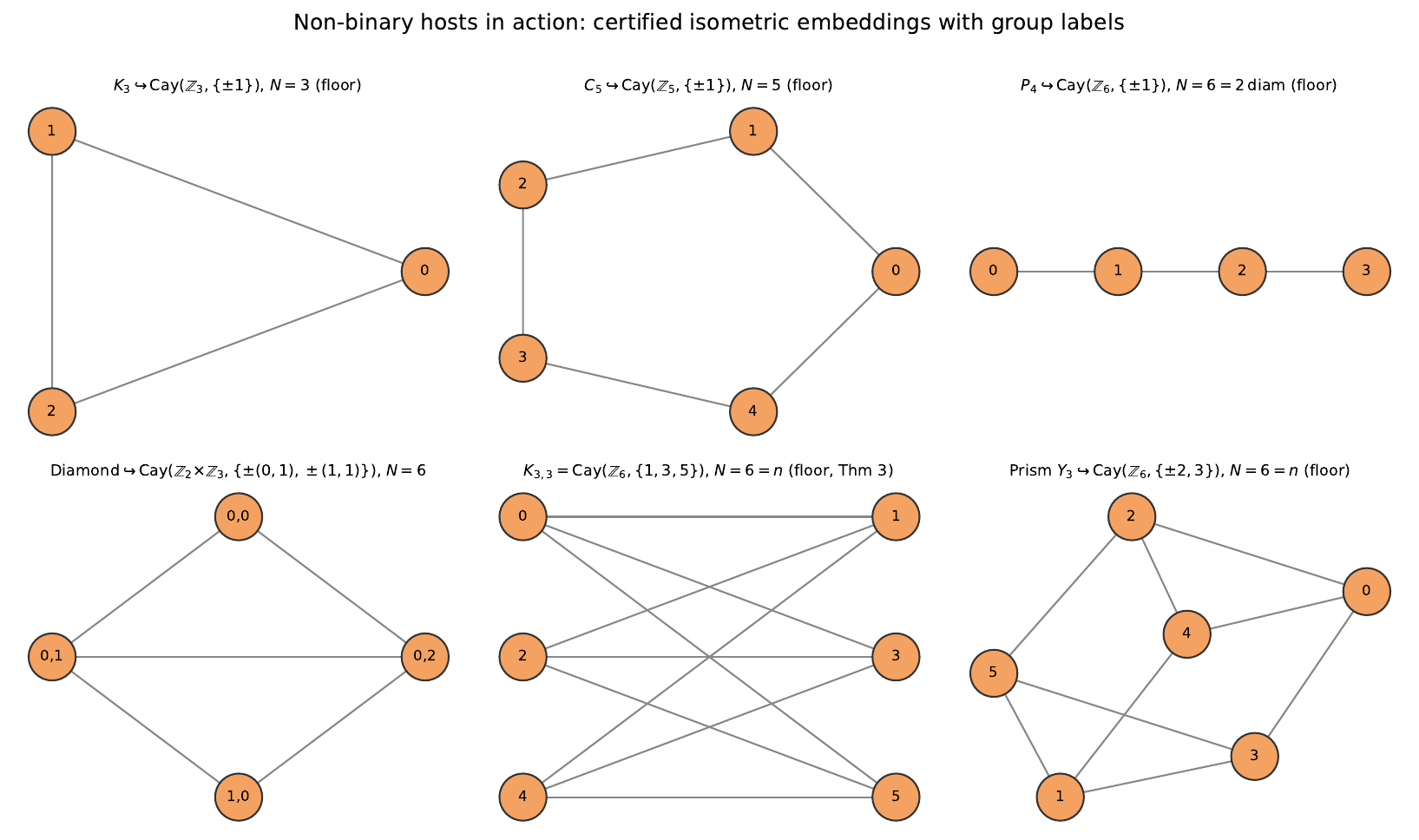}
\caption{Six certified non-binary embeddings with explicit group labels. Four
attain the order floor; the diamond illustrates a non-diagonal sublattice fold
\cite{fokam-p1}, and $K_{3,3}$ illustrates Theorem~\ref{thm:equality}.}
\label{fig:gallery}
\end{figure}

\begin{remark}[Structural reading]\label{rem:structure}
The universal group of a partition typically has free rank at least one:
whenever some class is crossed with nonzero net sign by some cycle, an
unbounded direction survives the quotient, and the compactification stage must
fold it into a finite cyclic factor $\Z_N$. The dividend arises because the
metrically admissible fold modulus $N$ is usually far below the $2^j$ that a
purely binary refolding would require: a free direction of label range $R$
folds isometrically at any modulus beyond $R+\diam(G)\gamma$
\cite{fokam-p1}, and moduli that are odd or composite are generically
available. Cyclic factors are thus not the signature of rare global symmetry;
they are the default product of folding free directions, and the graphs that
remain purely binary are those whose cycle structure is rich enough (many
short even cycles, odd-cycle interactions) to kill all free directions at the
quotient stage. The binary host retains a different and genuine distinction:
it is the \emph{universally guaranteed} construction, the one that never
fails and never needs a fold search (Corollary~1 of \cite{fokam-p1}), and it
remains optimal for a substantial minority of graphs. In the vocabulary of the
title: $\Z_2^k$ is the guaranteed floor of the theory, and the abelian
dividend---strict for a majority of small graphs---is what the general theory
pays on top of it.
\end{remark}

\begin{remark}[Scope and trend]\label{rem:trend}
Figure~\ref{fig:byn} shows a mild decline of the strict-dividend fraction from
$n=6$ ($65\%$) to $n=7$ ($56\%$). Whether the dividend persists at larger
orders, and how it interacts with density and girth, are open empirical
questions that the present census cannot settle; we state them as such rather
than extrapolate.
\end{remark}

\subsection{Benchmark families}

On structured inputs the dividend is large. Table~\ref{tab:benchmarks}
re-embeds several benchmark families. Where the graph is itself an abelian
Cayley graph, the value of $\nu$ is \emph{proved} by
Theorem~\ref{thm:equality} and the pipeline output is reported only as the
algorithmic upper bound; the case of $K_{3,3}=\Cay(\Z_6,\{1,3,5\})$ (a
circulant) is instructive, since the initializer portfolio does not propose
the circulant partition and the pipeline alone would report the weaker bound
$8$.

\begin{table}[t]
\centering
\caption{Benchmark families. ``bin'' is the binary pipeline's host order;
``best $\Gamma$'' the smallest certified abelian host; OPT marks host
$=\max(n,2\,\diam)$. Entries marked $^{\ast}$ are proved optimal by
Theorem~\ref{thm:equality}; the remaining entries are certified upper
bounds.}
\label{tab:benchmarks}
\begin{tabular}{lccccc}
\toprule
graph & $n$ & floor & bin & best $\Gamma$ & order\\
\midrule
ring $C_{12}$ & 12 & 12 & 64 & $\Z_{12}$ & 12\ \,OPT$^{\ast}$\\
ring $C_{16}$ & 16 & 16 & 256 & $\Z_{16}$ & 16\ \,OPT$^{\ast}$\\
path $P_{16}$ & 16 & 30 & 32768 & $\Z_{30}$ & 30\ \,OPT\\
circulant $C_{12}(1,2)$ & 12 & 12 & 64 & $\Z_{12}$ & 12\ \,OPT$^{\ast}$\\
grid $4\times 4$ & 16 & 16 & 64 & $\Z_6\times\Z_6$ & 36\\
$Q_3$ & 8 & 8 & 8 & $\Z_2^3$ & 8\ \,OPT$^{\ast}$\\
$K_{3,3}$ & 6 & 6 & 16 & $\Z_6$ & 6\ \,OPT$^{\ast}$\\
Petersen & 10 & 10 & 16 & $\Z_2^4$ & 16\\
Desargues & 20 & 20 & 32 & $\Z_2^5$ & 32\\
\bottomrule
\end{tabular}
\end{table}

\section{Exactly solved families and the Petersen problem}
\label{sec:solved}

Theorem~\ref{thm:equality} resolves $\nu$ for every graph that is itself an
abelian Cayley graph, and gives a clean obstruction for those that are not.

\begin{corollary}[Exactly solved families]\label{cor:solved}
\begin{itemize}
\item[(i)] Cycles: $\nu(C_m)=m$ for all $m\ge 3$; in particular the odd
cycles, which need binary host $2^{m-1}$, collapse to host $m$.
\item[(ii)] Complete graphs: $\nu(K_n)=n$, via
$K_n=\Cay(\Z_n,\Z_n\setminus\{0\})$.
\item[(iii)] Circulants: $\nu(C_n(d_1,\dots,d_r))=n$; in particular
$\nu(K_{3,3})=6$ via $K_{3,3}=\Cay(\Z_6,\{1,3,5\})$.
\item[(iv)] Paths: $\nu(P_k)=2(k-1)$, attained by
$P_k\hookrightarrow C_{2(k-1)}$; this upgrades the path-into-cycle stretching
rule to an optimality theorem.
\end{itemize}
\end{corollary}

\begin{proof}
(i)--(iii) are abelian Cayley graphs, so Theorem~\ref{thm:equality} applies.
(iv) The lower bound is $2\,\diam=2(k-1)$ by Theorem~\ref{thm:orderlb}, and
$C_{2(k-1)}$ contains $P_k$ isometrically.
\end{proof}

\begin{remark}[The Petersen problem]\label{rem:petersen}
The Petersen graph is vertex-transitive but, famously, not a Cayley
graph---indeed it is the smallest vertex-transitive non-Cayley graph
\cite{mckay-praeger1994}. By Theorem~\ref{thm:equality} the strict inequality
forces $\nu(\mathrm{Petersen})\ge n+1=11$, while the binary embedding into the
Clebsch graph \cite{fokam-p1} gives $\nu(\mathrm{Petersen})\le 16$.
Determining the exact value of $\nu(\mathrm{Petersen})\in[11,16]$ is, to our
knowledge, open.
\end{remark}

\section{Conclusion}
\label{sec:conclusion}

We have bounded the binary dimension and the host order of isometric
embeddings of graphs into abelian Cayley graphs, determined both exactly for
several families, and shown that the dimension window
$[\max(\diam,\lceil\log_2 n\rceil),\,n-1]$ is filled at both ends---by
hypercubes, complete graphs, even cycles, and the Petersen graph at the lower
end (via the geodesic-independence and injectivity bounds), by stars one step
above it (via maximum sum-free sets), and by odd cycles at the upper end (via
the cyclic-interval lemma). The exhaustive census of all $995$ connected
graphs on at most seven vertices exhibits the abelian dividend: a majority of
small graphs admit certified hosts strictly smaller than their best binary
hosts, cyclic factors appear in $71\%$ of the best hosts found, and only $17$
graphs sit exactly on the order floor. The binary host retains its role as
the universally guaranteed construction; the general abelian theory is what
converts that guarantee into strict compression for most graphs.

Three problems remain open: the cyclic interval lemma
(Conjecture~\ref{conj:interval}), which would give $\kmin(C_m)=m-1$ for all
odd $m$ unconditionally; the exact value of
$\nu(\mathrm{Petersen})\in[11,16]$; and the behaviour of the dividend as $n$
grows (Remark~\ref{rem:trend}). The construction underlying these bounds is
developed in the companion paper \cite{fokam-p1}, and the harmonic analysis
the embeddings support in the companion papers \cite{fokam-p3,fokam-p4}.

\section*{Data and code availability}
The full census records (one certified record per graph), the corrected
embedding pipeline, and the runner scripts are available from the authors and
will be archived with the final version. All graphs are taken from the atlas
of Read and Wilson \cite{read-wilson1998}.

\section*{Declaration on the use of AI}
An AI assistant (Anthropic's Claude) was used for software development and
debugging of the verification pipeline and for language editing. All
mathematical content, directions, and conclusions are the authors' own; all
computational claims were verified by the certified pipeline described in
Section~\ref{subsec:methodology}.

\section*{Acknowledgments}
The authors thank Solutum Engineering for internet and computing support.

\end{document}